\documentclass[12pt]{amsart}

\usepackage{amssymb}
\usepackage[top=2cm,bottom=5cm,left=3cm,right=1cm]{geometry}
\usepackage[hidelinks]{hyperref}

\theoremstyle{plain}
\newtheorem{theorem}{Theorem}[section]
\newtheorem{lemma}[theorem]{Lemma}
\newtheorem{proposition}[theorem]{Proposition}

\theoremstyle{definition}

\theoremstyle{remark}

\newcommand{\CT}{\operatorname{CT}}
\newcommand{\Z}{\mathbb{Z}}

\title[Groups of Class Transpositions]
{Groups of Class Transpositions with Prescribed Prime Divisors of the Moduli}
\author{Alex Iskra}
\address{Sobolev Institute of Mathematics, 4 Acad. Koptyug avenue, 630090, Novosibirsk, Russia.}
\address{Novosibirsk State  University, 2 Pirogova Street, 630090, Novosibirsk, Russia.} 
\email{a.iskra@g.nsu.ru}
\date{\today}

\begin{document}
	
	\begin{abstract}
		For a set $\mathcal P$ of odd primes, let $\CT_{\mathcal P}(\Z)$
		denote the group generated by all class transpositions whose
		moduli have no odd prime divisors outside $\mathcal P$.
		We prove that
		\[
		\bigl\langle
		\CT_{\mathcal P_1}(\Z),\CT_{\mathcal P_2}(\Z)
		\bigr\rangle
		=
		\CT_{\mathcal P_1\cup\mathcal P_2}(\Z)
		\]
		for any sets $\mathcal P_1$ and $\mathcal P_2$ of odd primes.
		This gives a negative answer to Question~21.75
		in the Kourovka Notebook.
	\end{abstract}
	
	\maketitle
	
	
	\section{Introduction}
	
	Class transpositions of the integers were introduced by S.~Kohl in 2010 \cite{K}. For $m\geqslant1$ and $0\leqslant a<m$, let	$a(m)=a+m\Z$ denote the residue class of $a$ modulo $m$. Two
	classes $a(m)$ and $b(n)$ are disjoint if and only if
	$a\not\equiv b\pmod{\gcd(m,n)}$.
	
	Let $a(m)$ and $b(n)$ be disjoint. The \textit{class transposition}
	associated with them is the permutation
	\[
	\tau_{a(m),b(n)}
	=
	\prod_{k\in\mathbb Z}
	(a+mk,b+nk) \in \operatorname{Sym}(\mathbb{Z}).
	\]
	
	The group generated by all class transpositions is denoted by $\CT(\Z)$.
	If $\mathcal P$ is a set of odd primes, then
	$\CT_{\mathcal P}(\Z)$ denotes the group generated by all class
	transpositions $\tau_{a(m),b(n)}$ for which every odd prime divisor
	of $mn$ belongs to $\mathcal P$. In particular, powers of two are
	allowed in the moduli for every $\mathcal P$.
	
	In 2026, the twenty-first edition of the Kourovka Notebook included
	Question~21.75 \cite{Kourovka}: \textit{given two distinct sets
	$\mathcal P_1$ and $\mathcal P_2$ of odd primes, neither of which
	is contained in the other, is it true that}
	\[
	\bigl\langle
	\CT_{\mathcal P_1}(\Z),\CT_{\mathcal P_2}(\Z)
	\bigr\rangle
	\lneq
	\CT_{\mathcal P_1\cup\mathcal P_2}(\Z)?
	\]
	It is clear that
	$\bigl\langle
	\CT_{\mathcal P_1}(\Z),\CT_{\mathcal P_2}(\Z)
	\bigr\rangle\leqslant\CT_{\mathcal P_1\cup\mathcal P_2}(\Z)$.
	In this paper, we prove that equality holds for arbitrary sets
	$\mathcal P_1$ and $\mathcal P_2$ of odd primes, thereby giving a
	negative answer to Question~21.75.
	
	The paper is organized as follows. Sections~\ref{sec:separated}
	and~\ref{sec:horizontal} treat separated and horizontal class transpositions.
	Section~\ref{sec:matching} proves a matching lemma using Hall's theorem.
	These results are combined in Section~\ref{sec:main-result}
	to prove the main theorem.

	\section{Class transpositions separated modulo a power of two}
	\label{sec:separated}
	
	Fix a power of two $R=2^k\geqslant4$. A class transposition
	$\tau_{a(m),b(n)}$ is called \textit{$R$-separated} if
	\[
	R\mid m,
	\qquad R\mid n,
	\qquad a\not\equiv b\pmod R.
	\]
	The last condition, in particular, ensures that the classes
	$a(m)$ and $b(n)$ are disjoint.
	
	\begin{lemma}
		\label{lem:remove-prime}
		Let the class transposition $\tau_{a(m),b(n)}$ be $R$-separated.
		Suppose that $q$ is an odd prime and $q\mid n$. Then there exists
		a class transposition $\tau_{b_0(Rq),c(R)}$ such that
		\begin{equation}\label{eq:remove-prime-conjugation}
		\tau_{b_0(Rq),c(R)}\tau_{a(m),b(n)}\tau_{b_0(Rq),c(R)}
		=\tau_{a(m),b'(n/q)}.
		\end{equation}
		The class transposition on the right-hand side is again
		$R$-separated.
	\end{lemma}
	
	\begin{proof}
		Since $R\mid n$, $q\mid n$, and $q$ is odd, the number $d=Rq$
		divides $n$. Divide $b$ by $d$ with remainder:
		\[
		b=b_0+dt,
		\qquad
		0\leqslant b_0<d.
		\]
		The inequality $0\leqslant b<n$ implies that
		\begin{equation}\label{eq:quotient-range}
		0\leqslant t<\frac nd.
		\end{equation}
		
		Since $a\not\equiv b\pmod R$ and $R\geqslant4$, we may choose
		$c\in\{0,\ldots,R-1\}$ such that
		\begin{equation}\label{eq:auxiliary-residue}
		c\not\equiv a\pmod R,
		\qquad
		c\not\equiv b\pmod R.
		\end{equation}
		Because $b_0\equiv b\pmod R$, the condition
		$a\not\equiv b\pmod R$ together with \eqref{eq:auxiliary-residue} shows that
		the classes
		\[
		a(m),\quad b_0(Rq),\quad c(R)
		\]
		are pairwise disjoint. Thus the class transposition
		$\tau_{b_0(Rq),c(R)}$ is well defined. It fixes the class $a(m)$
		pointwise.
		
		Let us compute the image of the class $b(n)$ under
		$\tau_{b_0(Rq),c(R)}$. For every $k\in\Z$, we have
		\begin{equation}\label{eq:class-image}
		\begin{aligned}
			b+nk
			&=b_0+Rq\left(t+\frac{n}{Rq}k\right)\\
			&\xrightarrow{\tau_{b_0(Rq),c(R)}}
			c+R\left(t+\frac{n}{Rq}k\right)\\
			&=(c+Rt)+\frac nq k.
		\end{aligned}
		\end{equation}
		Set $b'=c+Rt$. By \eqref{eq:quotient-range} and $0\leqslant c<R$, we obtain
		\[
		0\leqslant b'
		<R+R\left(\frac{n}{Rq}-1\right)
		=\frac nq.
		\]
		Since $\tau_{b_0(Rq),c(R)}$ fixes $a(m)$, it follows that
		\[
		\tau_{b_0(Rq),c(R)}\tau_{a(m),b(n)}\tau_{b_0(Rq),c(R)}
		=\tau_{a(m),b'(n/q)}.
		\]
		Finally,
		\[
		b'=c+Rt\equiv c\not\equiv a\pmod R,
		\]
		so the class transposition on the right remains $R$-separated.
	\end{proof}
	
	\begin{lemma}
		\label{lem:R-separated-in-H}
		Suppose that every odd prime divisor of $mn$ belongs to
		$\mathcal P_1\cup\mathcal P_2$. Then every $R$-separated class
		transposition $\tau_{a(m),b(n)}$ belongs to the group
		$\bigl\langle\CT_{\mathcal P_1}(\Z),
		\CT_{\mathcal P_2}(\Z)\bigr\rangle$.
	\end{lemma}
	
	\begin{proof}
		Suppose that $q\in\mathcal P_2\setminus\mathcal P_1$ divides one
		of the moduli. Applying Lemma~\ref{lem:remove-prime} (with the two
		classes interchanged if necessary), we conjugate by a class
		transposition whose moduli are $Rq$ and $R$, thereby decreasing the
		exponent of $q$ in that modulus by one. This conjugating class
		transposition belongs to $\CT_{\mathcal P_2}(\Z)$, since $q$ is the
		only possible odd prime divisor of its moduli.
		
		Continue this process. After each step, both moduli are still
		divisible by $R$, since $q$ is odd, and the representatives of the
		two classes remain distinct modulo $R$. After finitely many steps,
		we remove from both moduli all prime divisors belonging to
		$\mathcal P_2\setminus\mathcal P_1$. The result is an
		$R$-separated class transposition $\tau_0$ such that every odd
		prime divisor of its moduli belongs to $\mathcal P_1$. Hence
		\[
		\tau_0\in\CT_{\mathcal P_1}(\Z)\leqslant \bigl\langle
		\CT_{\mathcal P_1}(\Z),\CT_{\mathcal P_2}(\Z)
		\bigr\rangle.
		\]
		The original class transposition is conjugate to $\tau_0$ by a
		product of elements of
		$\CT_{\mathcal P_2}(\Z)\leqslant\bigl\langle
		\CT_{\mathcal P_1}(\Z),\CT_{\mathcal P_2}(\Z)
		\bigr\rangle$. Therefore it also belongs to
		$\bigl\langle\CT_{\mathcal P_1}(\Z),
		\CT_{\mathcal P_2}(\Z)\bigr\rangle$.
	\end{proof}
	
	\section{Horizontal class transpositions}
	\label{sec:horizontal}
	
	A class transposition $\tau_{a(m),b(n)}$ is called \textit{horizontal}
	if $m=n$.
	\begin{lemma}
		\label{lem:horizontal}
		Suppose that $R\mid M$ and that every odd prime divisor of $M$
		belongs to $\mathcal P_1\cup\mathcal P_2$. Then, for any distinct
		$x,y\in\{0,\ldots,M-1\}$, one has
		\[
		\tau_{x(M),y(M)}\in \bigl\langle
		\CT_{\mathcal P_1}(\Z),\CT_{\mathcal P_2}(\Z)
		\bigr\rangle.
		\]
	\end{lemma}
	
	\begin{proof}
		If $x\not\equiv y\pmod R$, apply
		Lemma~\ref{lem:R-separated-in-H}.
		
		Now suppose that $x\equiv y\pmod R$. Choose
		$z\in\{0,\ldots,R-1\}$ such that $z\not\equiv x\pmod R$. Since
		$R\mid M$, we have $0\leqslant z<M$. The classes
		$x(M),y(M),z(M)$ are pairwise distinct and hence pairwise
		disjoint. Therefore
		\begin{equation}\label{eq:horizontal-conjugation}
		\tau_{x(M),y(M)}
		=
		\tau_{x(M),z(M)}
		\tau_{y(M),z(M)}
		\tau_{x(M),z(M)}.
		\end{equation}
		Each factor on the right is $R$-separated and hence belongs to
		$\bigl\langle\CT_{\mathcal P_1}(\Z),
		\CT_{\mathcal P_2}(\Z)\bigr\rangle$ by
		Lemma~\ref{lem:R-separated-in-H}.
	\end{proof}
	
	\section{A matching lemma}
	\label{sec:matching}
	
	A \textit{matching} in a graph is a set of pairwise nonincident edges.
	It covers a vertex if that vertex is incident with one of its edges,
	and it is \textit{perfect} if it covers every vertex of the graph. For a
	bipartite graph with parts $X$ and $Y$ and a subset $S\subseteq X$,
	let $N(S)\subseteq Y$ denote the set of all vertices adjacent to at
	least one vertex of $S$. We recall Hall's classical marriage theorem.
	
	\begin{theorem}[Hall]
		Let $\Gamma$ be a finite bipartite graph with parts $X$ and $Y$.
		There exists a matching that covers every vertex in $X$ if and only
		if
		\[
		|N(S)|\geqslant |S|
		\]
		for every subset $S\subseteq X$. In particular, if $|X|=|Y|$ and
		this condition holds, then $\Gamma$ has a perfect matching.
	\end{theorem}
	
	\begin{lemma}
		\label{lem:matching}
		Let $a,b\in\Z$, $m,n\in\mathbb N$, and $R\in\mathbb N$. If
		\begin{equation}\label{eq:matching-condition}
		\gcd(m,R)+\gcd(n,R)\leqslant R,
		\end{equation}
		then there exists a permutation $\pi$ of the set
		$\{0,\ldots,R-1\}$ such that
		\begin{equation}\label{eq:matching-permutation}
		a+mj\not\equiv b+n\pi(j)\pmod R
		\qquad(0\leqslant j<R).
		\end{equation}
	\end{lemma}
	
	\begin{proof}
		For every $c\in\Z/R\Z$, set
		\[
		\begin{aligned}
			\alpha_c
			&=\bigl|\{j\in\{0,\ldots,R-1\}:a+mj\equiv c\pmod R\}\bigr|,\\
			\beta_c
			&=\bigl|\{j\in\{0,\ldots,R-1\}:b+nj\equiv c\pmod R\}\bigr|.
		\end{aligned}
		\]
		In particular,
		\[
		\alpha_c\leqslant\gcd(m,R),
		\qquad
		\beta_c\leqslant\gcd(n,R).
		\]
		By \eqref{eq:matching-condition}, we obtain
		\begin{equation}\label{eq:residue-multiplicities}
		\alpha_c+\beta_c\leqslant R
		\qquad(c\in\Z/R\Z).
		\end{equation}
		
		Construct a bipartite graph whose left and right parts are two
		copies of $\{0,\ldots,R-1\}$. Join a left vertex $j$ to a right
		vertex $k$ if and only if
		\[
		a+mj\not\equiv b+nk\pmod R.
		\]
		We verify Hall's condition.
		
		Let $S$ be a set of left vertices. Hall's condition is immediate
		for $S=\varnothing$, so assume that $S\ne\varnothing$. If the
		values $a+mj\pmod R$, $j\in S$, include at least two distinct
		residues, then every right vertex is adjacent to at least one
		vertex in $S$. Indeed, the residue corresponding to a right vertex
		cannot coincide with two distinct residues simultaneously. Hence,
		in this case, $N(S)$ is the entire right part and
		\[
		|N(S)|=R\geqslant|S|.
		\]
		
		If, on the other hand, all vertices in $S$ have the same residue
		$c$, then $|S|\leqslant\alpha_c$, while the neighbors of $S$ are
		precisely the right vertices whose residues differ from $c$.
		Consequently, using \eqref{eq:residue-multiplicities}, we obtain
		\[
		|N(S)|=R-\beta_c\geqslant\alpha_c\geqslant|S|.
		\]
		Thus Hall's condition holds for every $S$. The graph therefore has
		a perfect matching. The corresponding bijection
		$j\mapsto\pi(j)$ satisfies \eqref{eq:matching-permutation}.
	\end{proof}
	
	\section{The main theorem}
	\label{sec:main-result}
	
	The following proposition was proved in \cite{B}.
	\begin{proposition} \label{propeq}
		Let $\tau_{a(m),b(n)}$ be a class transposition. Then:
		\begin{enumerate}
			\item the permutation
			$\tau_{(a+mj)(Km),(b+nj)(Kn)}$ is a class transposition for all
			integers $j$ and $K$ such that $0 \leqslant j\leqslant K-1$;
			\item for every positive integer $K$, one has
			\[
			\tau_{a(m),b(n)}=
			\prod\limits_{j=0}^{K-1}
			\tau_{(a+mj)(Km),(b+nj)(Kn)}.
			\]
		\end{enumerate}
	\end{proposition}
	
	\begin{theorem}
		\label{thm:main}
		For any sets $\mathcal P_1$ and $\mathcal P_2$ of odd primes, one
		has
		\[
		\bigl\langle
		\CT_{\mathcal P_1}(\Z),\CT_{\mathcal P_2}(\Z)
		\bigr\rangle
		=
		\CT_{\mathcal P_1\cup\mathcal P_2}(\Z).
		\]
	\end{theorem}
	
	\begin{proof}
		Consider an arbitrary generator
		\begin{equation}\label{eq:generator}
		T=\tau_{a(m),b(n)}
		\in\CT_{\mathcal P_1\cup\mathcal P_2}(\Z).
		\end{equation}
		
		Choose $k$ sufficiently large that $R=2^k$ satisfies
		\begin{equation}\label{eq:power-of-two}
		R\geqslant4,
		\qquad
		\gcd(m,R)+\gcd(n,R)\leqslant R.
		\end{equation}
		
		By Proposition~\ref{propeq},
		\begin{equation}\label{eq:refinement}
		T=\prod_{j=0}^{R-1}
		\tau_{(a+mj)(Rm),(b+nj)(Rn)}.
		\end{equation}
		Set
		\[
		A_j=(a+mj)(Rm),
		\qquad
		B_j=(b+nj)(Rn).
		\]
		All factors in \eqref{eq:refinement} have pairwise disjoint supports.
		
		By Lemma~\ref{lem:matching}, there exists a permutation
		$\pi\in S_R$ such that
		\begin{equation}\label{eq:separated-pairing}
		a+mj\not\equiv b+n\pi(j)\pmod R
		\qquad(0\leqslant j<R).
		\end{equation}
		Consider the permutation
		\begin{equation}\label{eq:permutation-U}
		U=\prod_{j=0}^{R-1}\tau_{A_j,B_{\pi(j)}}.
		\end{equation}
		By \eqref{eq:separated-pairing}, every class transposition in this product is
		$R$-separated. Its moduli are $Rm$ and $Rn$, and their odd prime
		divisors belong to $\mathcal P_1\cup\mathcal P_2$. Therefore, by
		Lemma~\ref{lem:R-separated-in-H}, all factors in \eqref{eq:permutation-U}
		belong to
		$\bigl\langle\CT_{\mathcal P_1}(\Z),
		\CT_{\mathcal P_2}(\Z)\bigr\rangle$. Hence
		$U\in\bigl\langle\CT_{\mathcal P_1}(\Z),
		\CT_{\mathcal P_2}(\Z)\bigr\rangle$.
		
		We adopt the convention that, in a product of permutations, the
		rightmost factor acts first. Set $V=TU$. We now describe the action
		of $V$. If
		\[
		x=a+mj+Rm\ell\in A_j,
		\]
		then
		\[
		x\stackrel{U}{\longmapsto}
		b+n\pi(j)+Rn\ell
		\stackrel{T}{\longmapsto}
		a+m\pi(j)+Rm\ell.
		\]
		Consequently,
		\[
		V\colon A_j\longrightarrow A_{\pi(j)}
		\]
		with the parameter $\ell$ unchanged. Similarly, if $x\in B_j$,
		then
		\[
		V\colon B_j\longrightarrow B_{\pi^{-1}(j)}
		\]
		with the parameter unchanged. Outside $a(m)\cup b(n)$, both $T$
		and $U$ act as the identity.
		
		Write the permutation $\pi\in S_R$ as a product of transpositions:
		\[
		\pi=(i_1\,j_1)\cdots(i_t\,j_t).
		\]
		The action of $V$ described above therefore gives
		\begin{equation}\label{eq:horizontal-factorization}
		V=
		\bigl(
		\tau_{A_{i_1},A_{j_1}}\cdot\ldots\cdot
		\tau_{A_{i_t},A_{j_t}}
		\bigr)
		\bigl(
		\tau_{B_{i_t},B_{j_t}}\cdot\ldots\cdot
		\tau_{B_{i_1},B_{j_1}}
		\bigr).
		\end{equation}
		Indeed, the first factor induces the permutation $\pi$ on the
		classes $A_0,\ldots,A_{R-1}$ and fixes every class $B_j$ pointwise, whereas
		the second induces the permutation $\pi^{-1}$ on the classes
		$B_0,\ldots,B_{R-1}$ and fixes every class $A_j$ pointwise.
		
		Every transposition $\tau_{A_{i_\nu},A_{j_\nu}}$ is a horizontal
		class transposition of modulus $Rm$, and every
		$\tau_{B_{i_\nu},B_{j_\nu}}$ is a horizontal class transposition
		of modulus $Rn$. Both moduli are divisible by $R$, and all their
		odd prime divisors belong to $\mathcal P_1\cup\mathcal P_2$.
		Therefore, by Lemma~\ref{lem:horizontal}, all factors in
		\eqref{eq:horizontal-factorization} belong to
		$\bigl\langle\CT_{\mathcal P_1}(\Z),
		\CT_{\mathcal P_2}(\Z)\bigr\rangle$. Hence
		$V\in\bigl\langle\CT_{\mathcal P_1}(\Z),
		\CT_{\mathcal P_2}(\Z)\bigr\rangle$.
		
		Thus every generator of
		$\CT_{\mathcal P_1\cup\mathcal P_2}(\Z)$ belongs to
		$\bigl\langle\CT_{\mathcal P_1}(\Z),
		\CT_{\mathcal P_2}(\Z)\bigr\rangle$. Together with the evident
		reverse inclusion, this yields
		\[
		\bigl\langle
		\CT_{\mathcal P_1}(\Z),\CT_{\mathcal P_2}(\Z)
		\bigr\rangle=\CT_{\mathcal P_1\cup\mathcal P_2}(\Z),
		\]
		as required.
	\end{proof}
	
	\section*{Acknowledgements}
	
	The author thanks Stefan Kohl for checking the manuscript.
	
	\section*{Declaration of generative AI and AI-assisted technologies in the manuscript preparation process}
	
	During the preparation of this work, the author used ChatGPT-5.6 Sol to assist in drafting and checking parts of the technical arguments and in translating and editing the manuscript. Every AI-assisted mathematical argument was independently verified and revised by the author. The author made all final mathematical decisions and assumes full responsibility for the correctness and originality of the manuscript.

\end{document}